\documentclass{article}
\usepackage{amsmath,amsfonts,amsthm,amscd,amssymb}
\usepackage{pstricks}
\usepackage{caption}
\usepackage{graphicx}
\usepackage{tabulary}
\newtheorem{thm}{Theorem}[section]
\newtheorem{pro}[thm]{Proposition}

\newtheorem{defn}[thm]{Definition}

\begin{document}

\author{Mark Herman and Jonathan Pakianathan}
\title{On the distribution of distances in homogeneous compact metric spaces}
\maketitle

\begin{abstract}
We provide a simple proof that in any homogeneous, compact metric space of
diameter $D$, if one finds the average distance $A$ achieved in $X$ with respect to some
isometry invariant Borel probability measure, then
$$\frac{D}{2} \leq A \leq D.$$ This result applies equally to vertex-transitive graphs and to
compact, connected, homogeneous Riemannian manifolds.

We then classify the cases where one of the extremes occurs. In particular any homogeneous compact metric space where $A=\frac{D}{2}$ possesses a strict antipodal property which implies in particular that the distribution of distances in $X$ is symmetric about $\frac{D}{2}$ which is hence both mean and median of the distribution.

In particular, we show that the only closed, connected, \\ positive-dimensional Riemannian manifolds with this strict antipodal property are spheres.


\end{abstract}


\section{Introduction}
Good general references for the mathematics used in this paper are \cite{Rud} for measure theory, \cite{Mun} for general topology,
\cite{Hat} for algebraic topology and \cite{Lee} for Riemannian geometry.

This paper was motivated by a question of Alan Kaplan that appears on the webpage (www.math.uiuc.edu/$\sim$west/openp): For any finite, connected, vertex-transitive graph, is the average distance between vertices at least as large as half the diameter of the graph?

We answer this question in the affirmative by establishing:

\begin{thm}
\label{theorem: vertextransitiveintro}
Let $X$ be a finite, connected, vertex transitive graph with $n$ vertices and diameter $D$.
Let $A = \frac{1}{n^2}\sum_{x,y \in V} d(x,y)$ be the average distance between vertices in this graph, then
$$
\frac{D}{2} \leq A \leq (1-\frac{1}{n})D
$$
If we define the average distance instead via $\bar{A} = \frac{1}{n(n-1)}\sum_{x \neq y} d(x,y)$ we have
$$
\frac{D}{2}(\frac{n}{n-1}) \leq \bar{A} \leq D.
$$
\end{thm}

These bounds are sharp: For any complete graph $K_n$, $\bar{A}=D=1$. On the other hand,
for a $d$-dimensional hypercube graph $Q_d$, we have $D=d$, and
$A=\frac{1}{2^d} \sum_{k=0}^{d} k\binom{n}{k} = \frac{d2^{d-1}}{2^d}=\frac{d}{2}=\frac{D}{2}$.

A similar bound for the expected distance squared in a vertex transitive graph in terms of the diameter $D$ was obtained in \cite{NRa09}:
$\frac{D^2}{8} \leq E[d^2] \leq D^2.$ However, when this argument is applied to the expected distance it results in a weaker result than
Theorem~\ref{theorem: vertextransitiveintro}.

Theorem~\ref{theorem: vertextransitiveintro} follows from a more general theorem on
homogeneous compact metric spaces, i.e., compact metric spaces whose isometry group acts transitively on points. (Indeed we prove the theorem for antipodal compact metric spaces - see section~\ref{section: antipodal} for definitions.)

\begin{thm}
\label{thm: main} Let $X$ be a homogeneous compact metric space and $m$ a Borel probability measure on $X$ invariant under all isometries. Let $$A=\int_{X \times X} d(x,y) m(x)m(y)$$ be the average distance in the metric space with respect to the product measure $m \times m$ on $X \times X$,
then
$$
\frac{D}{2} \leq A \leq \mu D
$$
where $1-\mu$ is the $m \times m$ measure of the diagonal in $X \times X$. $\mu=1$ if points have $m$-measure zero in $X$ and $\mu=(1-\frac{1}{n})$ if $|X|=n$ is finite.
\end{thm}

Theorem~\ref{thm: main} applies to any connected, finite, vertex-transitive graph under the uniform probability measure on vertices. It also applies to any compact, connected, homogeneous Riemmannian manifold $M$. If $G$ is the isometry group of $M$ then $G$ is a Lie group and
there is closed subgroup $K$ and continuous bijection $\theta: G/K \to M$. The push forward of a suitably normalized Haar measure on $G$ hence provides a $G$-invariant measure $m$ on $M$. Examples include spheres with spherical measure under the round spherical metric or under
the Euclidean metric (the isometry group is $O(n)$ for either metric). Any compact Lie group
such as torii under the probability Haar measure and under a left invariant Riemmanian metric provide other examples for which Theorem~\ref{thm: main} applies. Other examples include Grassmann spaces of various flavors.

In the case of unit spheres $S^d=\{x \in \mathbb{R}^{d+1} | \| x \|=1 \}$ under the spherical metric and measure, one has $A=\frac{\pi}{2}, D=\pi$ and so $A=\frac{D}{2}$.

In the case of the $p$-adic integers, a compact metric Abelian group equipped with its Haar measure $m$ and the standard $p$-adic metric, it is easy to compute that $D=1$ and $A=\frac{p}{p+1}$. As the prime $p \to \infty$ we have $A \to D$.

We also characterize the homogeneous compact metric spaces which achieve the bounds in
Theorem~\ref{thm: main}.

\begin{thm}
\label{thm: extremals}
Let $X$ be a homogeneous compact metric space and $m$ a Borel probability measure on $X$ invariant under all isometries and let $A$ be the average distance in $X$, and $D$ the diameter of $X$. \\
If $A=\mu D$, then $X$ is finite and the metric is a scaling of the standard discrete metric
$$
d(x,y)=\begin{cases}
1 \text{ if } x \neq y \\
0 \text{ otherwise}
\end{cases}
$$
and $m$ is the uniform probability measure.

On the other hand, $A=\frac{D}{2}$ if and only if $X$ is a strictly antipodal space, i.e., every point $x$ has a unique antipode $O_x$ such that $d(x, O_x)=D$ with the additional conditions
that there is an isometry of $X$ taking $x$ to $O_x$ and such that for any $y \in X$ one has
$$
D=d(x,y) + d(y,O_x).
$$
Furthermore, in any strictly antipodal space, the antipodal map $O: X \to X$ taking each $x$ to its unique antipode is an isometry of $X$ which is a central element of order exactly two in
the isometry group of $X$. It equips $X$ with a free $\mathbb{Z}/2\mathbb{Z}$-action.
In such a space the distribution of distances is symmetric about the mean (which is hence also the median) $\frac{D}{2}$. In other words, if $0 \leq a \leq D$ then
$$
Pr(d(x,y) \leq a) = Pr( d(x,y) \geq D-a ).
$$
where $Pr$ is probability with respect to the $m \times m$ probability measure on $X \times X$.
\end{thm}

In the case of unit spheres $S^d$ with spherical measure and metric, the distances are
uniformly distributed in the interval $[0,\pi]$. The antipodal map is the standard antipodal map
$x \to -x$.
In the case of distances in the hypercube graph in dimension $d$, the distances are distributed
according to the binomial distribution with parameters $n=d$ and $p=\frac{1}{2}$. The antipodal
map is the map that swaps zeros and ones in binary strings.
In the case of distances in a cycle graph $C_{2n}$, the average distance is
$A=\frac{1}{2n}(1(0) + 2(1) + 2(2) + \dots + 2(n-1) + 1(n))=\frac{n}{2}=\frac{D}{2}$.
The probability distribution of distances is
$$
Pr(d=j) =\begin{cases} \frac{1}{n} \text{ if } 1 \leq j \leq n-1 \\
\frac{1}{2n} \text{ if } j=0, n
\end{cases}
$$

Finally, we show in Theorem~\ref{thm:spheres} that the only closed, connected, \\ positive-dimensional Riemannian manifolds with this strict antipodal property are spheres.

\section{Proof of Theorem~\ref{thm: main}}
\label{section: antipodal}

We first define various concepts of antipodal spaces which will be useful.

\begin{defn} Let $X$ be a compact metric space of diameter $D$.
We say that $X$ is {\bf antipodal} if for every $x \in X$ there is at least one antipode
$O_x$ such that $d(x,O_x)=D$. We further require that there is an isometry of $X$ taking
$x$ to $O_x$.

We say that $X$ is {\bf uniquely antipodal} if it is antipodal and if each $x \in X$
has a unique antipode $O_x$.

Finally we say that $X$ is {\bf strictly antipodal} if it is antipodal and if for any $x \in X$, and
antipode $O_x$ of $x$ we have:
$$
D=d(x,y)+d(y,O_x)
$$
for all $y \in X$.
\end{defn}

Note that any homogeneous compact metric space is antipodal. This is because compactness guarantees a pair of points $u,v$ such that $d(u,v)=D$. Then the transitivity of the action of the isometry group on points, guarantees that every $x \in X$ has at least one $O_x$ such that
$d(x,O_x)=D$. Furthermore transitivity also guarantees an isometry taking $x \to O_x$ and
hence ensures the property of being an antipodal space.

\begin{thm}
\label{thm: mainantipodal} Let $X$ be a compact antipodal metric space of diameter $D$. Let
$m$ be a Borel probability measure on $X$ invariant under isometries and let
$$
A=\int_{X \times X} d(x,y) m(x)m(y).
$$
Then $\frac{D}{2} \leq A \leq \mu D$
where $\mu = 1 - (m\times m)(\Delta)$ where $\Delta$ is the diagonal in $X \times X$.
\end{thm}
\begin{proof}
First note $A = \int_{X \times X - \Delta} d(x,y)m(x)m(y) \leq D (m\times m)(X \times X - \Delta)$
and so $A \leq \mu D$.

On the other hand as $X \times X$ is compact and $d: X \times X \to \mathbb{R}$ continuous, we have
$d \in L^1(m \times m)$ and so we may apply Fubini's Theorem:
$$
A = \int_X (\int_X d(x,y) m(x) ) m(y).
$$
Thus to establish the lower bound, it is enough to show that
$A_y = \int_X d(x,y) m(x) \geq \frac{D}{2}$ for all $y \in X$.

To do this let $O_y$ be an antipode of $y$, as there is an isometry $g$ which takes
$y$ to $O_y$, and as $m$ is isometry invariant, we have $A_y=A_{O_y}$.

Then

$$
A_y = \int_X d(x,y) m(x) \geq \int_X (d(y,O_y) - d(O_y,x)) m(x)
= D - A_{O_y}.
$$
Thus $2A_y \geq D$ for any $y \in Y$ and we are done.

\end{proof}

When $X$ is a finite set of size $n$ and $m$ is uniform measure, the factor $\mu$ in
Theorem~\ref{thm: mainantipodal} is $\mu=(1-\frac{1}{n})$. On the other hand if points have
$m$-measure zero in $X$ then the diagonal $\Delta$ has $m \times m$ measure zero in
$X \times X$ by Fubini's Theorem and so $\mu=1$.

We have already discussed examples in the intro that show the bounds in Theorem~\ref{thm: mainantipodal} are sharp. We now characterize the examples which achieve the extremes of these bounds at least in the case of compact homogeneous metric spaces.

\begin{pro}
Let $X$ be a compact homogeneous metric space of diameter $D$ and $m$ an isometry invariant Borel probability measure on $X$. Suppose additionally that the average distance $A$ satisfies $A=\mu D$
where $\mu=1-(m \times m)(\Delta)$.

Then $X$ is a finite set equipped with a scalar multiple of the discrete metric and $m$ is
uniform measure.
\end{pro}
\begin{proof}
First let us handle the case $|X|=n$ is finite. Then as $X$ is homogeneous and $m$ is an isometry invariant Borel probability measure, it must assign every point of $X$ measure
$\frac{1}{n}$ and hence $\mu=1-\frac{1}{n}$. Writing $X=\{x_1,\dots,x_n\}$ we have
$$
A=\frac{1}{n^2} \sum_{j,k} d(x_j,x_k) = \mu D
$$
if and only if $d(x_j,x_k)=D$ when $j \neq k$ i.e. if the metric $d$ is $D$ times the discrete metric
on $X$.

On the other hand, when $X$ is infinite, points must have $m$-measure zero as every point has
equal measure by homogeneity and $m(X) = 1 < \infty$. Fubini's Theorem then shows
$(m \times m)(\Delta)=0$ and so $\mu=1$.

Now if $A=\mu D=D$ in this case we would have
$$
\int_{X \times X} d(x,y) m(x)m(y) = \int_X A_y m(y) = D.
$$
By homogeneity $A_y$ is independent of $y \in X$ and so this equation implies
$A_y = \int_X d(x,y)m(x)=D$ for all $y \in Y$.

Thus for any $y \in Y$, the set $S_y=\{ x \in X | d(x,y) < D \}$ has $m$-measure zero.
As $d: X \times X \to \mathbb{R}$ is continuous, $S_y$ is an open neighborhood of $y$.
The collection $\{ S_y | y \in X \}$ is an open cover of $X$ by open sets of $m$-measure zero.
As $X$ is compact, it is covered by a finite number of these and hence has $m$-measure zero
which contradicts that $m$ is a probability measure. Thus it is impossible for $A=\mu D$
when $X$ is infinite.
\end{proof}

\begin{pro}
\label{pro: lower}
Let $X$ be a compact homogeneous metric space of diameter $D$. Let $m$ be an isometry invariant Borel probability measure on $X$. Then the average distance $A$ satisfies
$A=\frac{D}{2}$ if and only if $X$ is a strictly antipodal space.
\end{pro}
\begin{proof}
Again $A_y = \int_X d(x,y) m(x)$ is independent of $y \in X$ by homogeneity and invariance of
$m$ under isometries. Thus $A=A_y$ for all $y \in X$ so if $A=\frac{D}{2}$ we have
$A_y = \int_X d(x,y) m(x) = \frac{D}{2}$. As $X$ is homogeneous, it is antipodal so let $O_y$ be an antipode of $y$. Then the inequality:
$$
A_y = \int_X d(x,y) m(x) \geq \int_X (d(y,O_y)-d(O_y,x))m(x) = D - A_{O_y}
$$
is an equality and so we conclude that the set
$T_y = \{ x \in X | d(x,y) > d(y,O_y) - d(O_y,x) \}$ is a set of $m$-measure zero.
If this set were nonempty, it would provide an $m$-measure zero open neighborhood of some
point in $X$. By homogeneity, every point of $X$ would have a measure zero open neighborhood.
Compactness of $X$ would then imply a finite open cover of $X$ by measure zero open sets contradicting that $m$ is a probability measure on $X$. Thus we conclude
$A=\frac{D}{2}$ implies that $d(y,O_y)=D=d(y,x)+d(x,O_y)$ for all $x, y \in X$ i.e., that
$X$ is a strictly antipodal space. Conversely it is easy to check that if $X$ is strictly antipodal,
then $A=\frac{D}{2}$.
\end{proof}

In the next section, we further study the structure of strictly antipodal spaces and show that the
distribution of distances in such spaces is symmetric around the mean distance.

\section{Strictly antipodal spaces}

In this section we study the structure of strictly antipodal spaces. By Proposition~\ref{pro: lower},
these are the primary class of compact metric spaces whose average distance (with respect to any isometry invariant Borel probability measure) is equal to half their diameter.

\begin{thm}
\label{thm:strictantipode}
Let $X$ be a strictly antipodal compact metric space. Then: \\
(1) $X$ is uniquely antipodal. \\
(2) The antipodal map $O: X \to X$ which takes each element $x$ to its unique antipode $O_x$
is an isometry and is a central element of order two in the isometry group of $X$. \\
(3) $X$ has a free $\mathbb{Z}/2\mathbb{Z}$-action via this antipode. Thus if its Euler characteristic is defined, it is even. \\
(4) If $m$ is a Borel probability measure on $X$ invariant under isometries, then the distribution
of distances in $X$ is symmetric about its average $\frac{D}{2}$ which is hence also the median distance. More precisely for any $0 \leq a \leq D$ we have
$Pr( d(x,y) \leq a) = Pr(d(x,y) \geq D-a)$
where $Pr$ is probability with respect to the measure $m \times m$ on $X \times X$.
\end{thm}
\begin{proof}
Part(1):
Let $x \in X$ then if $O_x$ is an antipode of $X$, strict antipodality says
$D = d(x,y) + d(y, O_x)$ for all $y \in X$. If $y$ were another antipode of $x$ then
$d(x,y)=D$ also and so $d(y,O_x)=0$ so $y=O_x$. Thus each element has a unique antipode and
$X$ is a uniquely antipodal space.

Part(2): Let $x,y \in X$. From strict antipodality $D = d(x,y) + d(y, O_x)$ and $D = d(O_x, O_y) + d(y, O_x)$. This implies $d(x, y)=d(O_x, O_y)$ and so the antipodal map $O: X \to X$ taking each $x$ to its unique
antipode $O_x$ is an isometry. $O \circ O =Id$ so it represents an element of order two in
$Iso(X)$, the group of isometries of $X$. If $g$ is any isometry of $X$ then
$d(gx,gO_x)=d(x,O_x)=D$ which implies $gO_x=O_{gx}$ by uniqueness of antipodes.
Thus $g \circ O = O \circ g$ and hence $O$ commutes with any isometry of $X$ in $Iso(X)$, i.e.,
it lies in the center of $Iso(X)$.

Part(3): This part follows immediately from (2) and basic facts about Euler characteristics from
algebraic topology.

Part(4): As $O: X \to X$ is an isometry it is $m$-measure preserving. Thus
$$Pr(d(x,y) \leq a) = \int_X (\int_{\{y | d(y,x) \leq a\}} m(y))m(x)
=\int_X (\int_{\{y | d(y,O_x) \leq a \}} m(y)) m(x) $$
as $O$ exchanges the sets $\{y | d(y,x) \leq a\}$ and $\{ y |d(y,O_x) \leq a\}$.
Finally strict antipodality lets us write the condition
$d(y,O_x) \leq a$ as $d(x,y) \geq D-a$. This completes the proof.

\end{proof}

We now show that the only closed, connected, positive dimensional Riemannian manifolds
which are strictly antipodal are spheres.

\begin{thm}
\label{thm:spheres}
Let $X$ be a closed, connected, Riemannian manifold which is strictly
antipodal. Then if $d=dim(X) \geq 1$ then $X$ is homeomorphic to $S^d$.
\end{thm}
\begin{proof}
Fix a point $p \in X$. As $X$ is closed and connected, it is geodesically complete and so the exponential map
$e: T_p(X) \to X$ is onto, continuous and smooth away from the cut locus of $p$.

If $D$ is the diameter of $X$, then the restriction of $e$ to the closed ball $B$ of radius $D$ in
$T_p(X)$ has $e: B \to X$ still onto.

If $y$ is a point of $X$ other than $p$ or $O_p$, the antipode of $p$, then $0 < d(p,y) < D$ and as
$D=d(p,y)+d(y,O_p)$ one sees that the concatenation of any minimal geodesic from $p$ to $y$
with any minimal geodesic from $y$ to $O_p$ has length $D=d(p,O_p)$ and hence is a path of minimal length from $p$ to $O_p$. However it is a standard fact that paths of minimum length between points in a geodesically complete Riemannian manifold are unbroken geodesics. Thus any minimal geodesic from $p$ to $y$ (of unit speed) must have matching velocity vector
with any minimal geodesic from $y$ to $O_p$ at the point $y$.

From this it follows that there is a unique minimal geodesic from $p$ to $y$ and from $y$ to
$O_p$ as if there were more than one of either, one could find a path of length $D$ from $p$ to $O_p$ consisting of a strictly broken geodesic path which contradicts $d(p,O_p)=D$.

This then implies that $e$ is injective when restricted to the interior of $B$ and so in particular,
$d(p,e(v)) = \|v \|$ for all $v \in Int(B)$. By continuity, this implies $d(p,e(v))=D$ for all
points $v$ on the boundary of $B$. However $O_p$ is the unique point of distance $D$ from $p$ and so we conclude that $e$ takes all of the boundary of $B$ to the single point $O_p$.

This implies that $e$ induces a bijective continuous map $\bar{e}: B/\sim \to X$
where $B/\sim$ is the closed ball with its boundary identified to a point. It is well-known that
$B/\sim$ is homeomorphic to $S^d$. Finally as $S^d$ is compact, and $X$ is Hausdorff, $\bar{e}$ is a homeomorphism. Furthermore it is smooth outside a single point corresponding to
the identification point of $S^d$.

\end{proof}

\end{document}